\documentclass[12pt,reqno]{amsart}
\usepackage{amsmath}
\usepackage{amsfonts}
\usepackage{amssymb}
\usepackage{enumerate}
\usepackage{amstext}
\usepackage{amsbsy}
\usepackage{amsopn}
\usepackage{bbm,amsthm}
\usepackage{amscd}
\usepackage[pdftex]{color}
\usepackage{amsxtra}
\usepackage{upref}
\usepackage{graphicx}
\usepackage{epstopdf}
\usepackage[usenames,dvipsnames,svgnames,table]{xcolor}
\usepackage[OT2,OT1]{fontenc}
\newcommand\cyr{%
\renewcommand\rmdefault{wncyr}%
\renewcommand\sfdefault{wncyss}%
\renewcommand\encodingdefault{OT2}%
\normalfont
\selectfont}
\DeclareTextFontCommand{\textcyr}{\cyr}

\DeclareFontFamily{OML}{rsfs}{\skewchar\font'177}
\DeclareFontShape{OML}{rsfs}{m}{n}{ <5> <6> rsfs5 <7> <8> <9> rsfs7
  <10> <10.95> <12> <14.4> <17.28> <20.74> <24.88> rsfs10 }{}
\DeclareMathAlphabet{\mathfs}{OML}{rsfs}{m}{n}

\newtheorem{thm}{Theorem}[section]
\newtheorem{lem}[thm]{Lemma}
\newtheorem{prop}[thm]{Proposition}

\newtheorem*{theorem*}{Theorem}

\newtheorem*{example*}{Example}

\numberwithin{equation}{section}

\renewcommand{\mod}{\mbox{$\,\mathrm{mod}\,$}}

\renewcommand{\epsilon}{\varepsilon}
\def\wt{\widetilde}

\def\text#1{\textrm{#1}}
\def\dsum{\sum}
\def\emptyset{\varnothing}
\def\dfrac#1#2{\frac{#1}{#2}}

\def\vf{\varphi}

\def\b{\beta}

\def \R{\mathbb R}
\def \N{{\mathbb N}}
\def\E{\mathbb E}
\def \Z{\mathbb Z}
\def \T{\mathbb T}

\def\ov{\overline}
\def\un{\underline}

\def\Q{\mathbb Q}

\def\cA{{\mathcal{A}}}
\def\cF{{\mathcal{F}}}

\def\({\biggl(}
\def\){\biggr)}

\def\<{\bold\langle}
\def\>{\bold\rangle}

\def\eps{{\varepsilon}}
\def\Prob{{\mathbb{P}}}

\DeclareMathOperator\const{const}

\DeclareMathOperator\dist{dist}

\DeclareMathOperator\sgn{sgn}

\def\DS{\displaystyle}

\title[No temporal DLT for a.e. irrational translation]{No temporal distributional limit theorem for a.e. irrational translation}
\author{Dmitry Dolgopyat and Omri Sarig }
\address{Department of Mathematics\\ University of Maryland at College Park, College Park, MD 20740, USA}
\email{dmitry@math.umd.edu}
\address{Faculty of Mathematics and Computer Science\\ The Weizmann Institute of Science\\ POB 26, Rehovot, Israel}
\email{omsarig@gmail.com}
\thanks{This work was partially supported by the ISF grant 199/14 and the BSF grant 2016105. 
The authors thank Adam Kanigowski for simplifying the statement and the proof of Lemma \ref{LmDV}. }

\def\tilde{\widetilde}
\def\bar{\overline}

\def\dist{{\rm dist}}
\def\eps{{\varepsilon}}

\def\mes{{\rm mes}}

\def\Card{{\rm Card}}

\def\Prob{{\mathbb{P}}}

\def\Tor{\mathbb{T}}

\def\bra{{\bar a}}

\def\brb{{\bar b}}

\def\brl{{\bar l}}

\def\brq{{\bar q}}

\def\brp{{\bar p}}

\def\brq{{\bar q}}

\def\breps{{\bar\eps}}

\def\cA{\mathcal{A}}

\def\cB{\mathcal{B}}

\def\cC{\mathcal{C}}

\def\cF{\mathcal{F}}

\def\cG{\mathcal{G}}

\def\cE{\mathcal{E}}

\def\cN{\mathcal{N}}

\def\fS{\mathfrak{S}}

\def\fT{\mathfrak{T}}

\def\fb{\mathfrak{b}}

\def\hK{{\hat K}}

\def\heps{{\hat{\eps}}}

\def\tK{{\tilde K}}

\makeatother

\def\beq{\begin{equation}}
\def\eeq{\end{equation}}

\begin{document}
\begin{abstract}
Bromberg and Ulcigrai constructed piecewise smooth functions on the torus such that the set of $\alpha$ for which the sum $ \sum_{k=0}^{n-1}f(x+k\alpha \mod 1)$ satisfies a temporal distributional limit theorem along the orbit of a.e. $x$ has  Hausdorff dimension one. We show that the Lebesgue measure of this set is equal to zero.
\end{abstract}

\maketitle
\section{Introduction and statement of main result}
\subsection{Background}
Suppose $T:X\to X$ is a map, $f:X\to\R$ is a function, and $x_0\in X$ is a {\em fixed}  initial condition. We say that the {\em $T$--ergodic sums  }$S_n=f(x_0)+f(Tx_0)+\cdots+f(T^{n-1} x_0)$ satisfy a {\em temporal distributional limit theorem (TDLT) on the orbit of $x_0$}, if there exists a non-constant real valued random variable $Y$, centering constants $A_N\in\R$ and scaling constants $B_N\to\infty$  s.t.
\begin{equation}\label{TDLT}
\frac{S_n-A_N}{B_N}\xrightarrow[N\to\infty]{} Y\text{ in distribution},
\end{equation}
when $n$ is sampled uniformly from $\{1,\ldots,N\}$ and $x_0$ is fixed. Equivalently, for every Borel set $E\subset\R$ s.t. $\Prob(Y\in\partial E)=0$,
$$
\frac{1}{N}\Card\{1\leq n\leq N: \frac{S_n-A_N}{B_N}\in E\}\xrightarrow[N\to\infty]{}\Prob(Y\in E).
$$
We allow and expect $A_N, B_N, Y$ to depend on $T,f,x_0$.

Such limit theorems have been discovered for several zero entropy uniquely ergodic transformations, including systems  where the more traditional spatial limit theorems, with $x_0$ is sampled from a measure on $X$, fail \cite{Beck-Giant-Leap-1,Beck-Giant-Leap-2, Avila-Dolgopyat-Duryev-Sarig, Dolgopyat-Sarig-JSP, Paquette-Son, Dolgopyat-Sarig-Horocycle-Winding}. Of particular interest are TDLT for
$$
R_\alpha:[0,1]\to [0,1],\    R_\alpha(x)=x+\alpha\mod 1,\  f_\beta(x):=1_{[0,\beta)}(x)-\beta,
$$
because the $R_\alpha$--ergodic sums of $f_\beta$ along the orbit of $x$ represent the discrepancy of the sequence $x+n\alpha\mod 1$ with respect to  $[0,\beta)$ \cite{Schmidt-Irregularity,Conze-Keane-Cylinder-Flow,Beck-Giant-Leap-1}. Another source of interest is  the connection to the ``deterministic random walk" \cite{Aaronson-Keane, Avila-Dolgopyat-Duryev-Sarig}.

The validity of the TDLT for $R_\alpha$ and $f_\beta$ depends on the diophantine properties of $\alpha$ and $\beta$.
Recall that $\alpha\in (0,1)$ is {\em badly approximable} if for some $c>0$, $|q\alpha-p|\geq c/|q|$ for all irreducible fractions $p/q$. Equivalently, the digits in the continued fraction expansion of $\alpha$ are bounded \cite{Khintchine-Continued-Fractions}.  Say that  $\beta\in (0,1)$ is {\em badly approximable with respect to $\alpha$} if for some $C>0$,
$
|q\alpha-\beta-p|>C/|q|\text{ for all }p,q\in\Z, q\neq 0.
$
If $\alpha$ is badly approximable then every $\beta\in \mathbb Q\cap (0,1)$ is  badly approximable with respect to $\alpha$. The recent paper \cite{Bromberg-Ulcigrai} shows:

\begin{thm}[Bromberg \& Ulcigrai]
Suppose $\alpha$ is badly approximable and $\beta$ is badly approximable with respect to $\alpha$, e.g. $\beta\in\Q\cap (0,1)$. Then  the $R_\alpha$-ergodic sums of $f_\beta$ satisfy a temporal distributional limit theorem with Gaussian limit on the orbit of every initial condition.
\end{thm}

The set of badly approximable $\alpha$  has  Hausdorff dimension one \cite{Jarnik}, but  Lebesgue measure zero \cite{Khintchine-A.S.-approximable}. This leads to the following
 question:
 {\em Is there a $\beta$ s.t. the $R_\alpha$--ergodic sums of $f_\beta$ satisfy a temporal distributional limit theorem for a.e. $\alpha$ and a.e. initial condition?}

In this paper we answer  this question negatively.

\subsection{Main result}

To state our result in its most general form,  we need the following terminology.

Let $\Tor:=\R/\Z$. We say that  $f:\Tor\to\R$ is {\em piecewise smooth} if there exists a finite set
$\fS\subset\Tor$ s.t. $f$ is continuously differentiable on $\Tor\setminus\fS$ and $\exists\psi:\Tor\to\R$ with bounded variation s.t. $f'=\psi$ on $\Tor\setminus\fS$.
For example: $f_\beta(x)=1_{[0,\beta)}(x)-\beta$ (take $\fS=\{0,\beta\}$, $\psi\equiv 0$). We show:

\begin{thm}
\label{ThBeckConverse}
Let $f$ be a  piecewise smooth  function of zero mean.
Then there is a set   of full measure $\cE\subset \T\times \T$ s.t. if $(\alpha, x)\in \cE$ then
the $R_\alpha$--ergodic sums of $f$  do not satisfy a TDLT on the orbit of $x$.
\end{thm}
\noindent
The condition $\int_\Tor f=0$ is necessary: By Weyl's equidistribution theorem, for every $\alpha\not\in\mathbb Q$, $f$ Riemann integrable s.t. $\int_\Tor f =1$, and   $x_0\in\Tor$, $S_n/N\xrightarrow[N\to\infty]{\text{dist}}\mathrm{U}[0,1]$ as  $n\sim \mathrm{U}(1,\ldots,N)$. See \S\ref{Section-Notation} for the notation.

This paper has a companion   \cite{Dolgopyat-Sarig-Quenched} which  gives a different proof of Theorem \ref{ThBeckConverse}, in the special case $f(x)=\{x\}-\frac{1}{2}$. Unlike the proof given below, \cite{Dolgopyat-Sarig-Quenched} does not identify the set of $\alpha$ where the TDLT fails, but it does give more information on the different scaling limits  for the distributions of $S_n$, $n\sim \text{U}(1,\ldots,N_k)$ along different subsequences  $N_k\to\infty$. \cite{Dolgopyat-Sarig-Quenched} also shows  that if we randomize both $n$ and $\alpha$
by sampling $(n,{\alpha})$ uniformly from $\{1,\ldots,N\}\times\Tor$, then $\bigl(S_n-\frac{1}{N}\sum_{k=1}^N S_k\bigr)/\sqrt{\ln N}$ converges in distribution to the Cauchy distribution.

The methods of \cite{Dolgopyat-Sarig-Quenched} are specific for $f(x)=\{x\}-\frac{1}{2}$,
and we do not know how to apply them to other functions such as   $f_\beta(x)=1_{[0,\beta)}(x)-\beta$.

\subsection{The structure of the proof}
Suppose $f$ is piecewise smooth and has mean zero.

We shall see below that if $f$  is continuous, then for a.e. $\alpha$,  $f$ is an $R_\alpha$-coboundry, therefore $S_n$ are bounded, hence \eqref{TDLT} cannot hold with $B_N\to\infty$,  $Y$ non-constant. We remark that
\eqref{TDLT} does hold  with  $B_N\equiv 1$, $A_N=f(x_0)$, $Y=$ distribution of minus the transfer function, but  this is not  a  TDLT since no actual scaling is involved.

The heart of the proof is to show that if $f$ is discontinuous,
then for a.e. $\alpha$, the temporal distributions of the ergodic sums have different asymptotic scaling behavior on different subsequences. The proof of this has three independent parts:
\begin{enumerate}[(1)]
\item A reduction to the case
$
f(x)=\dsum\limits_{m=1}^d b_m h(x+\beta_m)$,  $h(x):=\{x\}-\frac{1}{2}.
$
\item A proof that if  $\cN\subset\N$ has positive lower density, then there exists $M\geq 1$ s.t. the following set has full Lebesgue measure in $(0,1)$:
$$
\cA(\cN,M):=\left\{\alpha\in (0,1): \begin{array}{l}
\exists n_k\uparrow\infty,\; r_k \leq M\text{ s.t. for all $k$:} \\
r_k q_{n_k}\in\cN\text{, } {a_{n_k+1}}/{(a_1+\cdots+a_{n_k})}\to\infty
\end{array}
\right\}.
$$
 Here $a_n$ and $q_n$ are the partial quotients and principal denominators of $\alpha$, see \S\ref{Section-cfe}.
 \medskip
\item Construction of $\cN=\cN(b_1,\ldots,b_d;\beta_1,\ldots,\beta_d)\subseteq \N$ with positive density, s.t. for every  $\alpha\in\cA(\cN,M)$ and a.e. $x$,   one can analyze the  temporal distributions of the Birkhoff sums of $\sum\limits_{m=1}^d b_m h(x+\beta_m)$.
\end{enumerate}

\subsection{Notation}\label{Section-Notation}  $n\sim \text{U}(1,\ldots,N)$ means that $n$ is a random variable taking values in $\{1,\ldots,N\}$, each with probability $\frac{1}{N}$. $\text{U}[a,b]$ is the uniform distribution on $[a,b]$. Lebesgue's measure is denoted by $\text{mes}$.  $\N=\{1,2,3,\ldots\}$ and $\N_0=\N\cup\{0\}$. If $x\in\R$, then $\|x\|:=\dist(x,\Z)$ and $\{x\}$ is the unique number in $[0,1)$ s.t. $x\in\{x\}+\Z$. $\Card(\cdot)$ is the cardinality. If $\eps>0$, then $a=b\pm\eps$ means that $|a-b|\leq \eps$.

\section{Reduction to the case $f(x)=\sum\limits_{m=1}^d b_m h(x+\beta_m)$}
\label{ScPrel}

 Let  $h(x)=\{x\}-\frac{1}{2}$, and let $\cG$ denote the collection of all non-identically zero functions of the form
 $f(x)=\sum\limits_{m=1}^d b_m h(x+\beta_m),$ where $d\in\N, b_i,\beta_i\in \R.$
  We explain how  to reduce the proof of Theorem \ref{ThBeckConverse} from the case of a general  piecewise smooth $f(x)$ to the case $f\in\cG$.

The following proposition was proved in  \cite{Dolgopyat-Sarig-Quenched}. Let $C(\Tor)$ denote the space of continuous real-valued functions on $\Tor$ with the sup norm.
\begin{prop}\label{PrPerNeg}
If $f(t)$ is differentiable on $\Tor\setminus\{\beta_1,\ldots,\beta_d\}$ and $f'$ extends to a function with bounded variation on $\Tor$, then there are $d\in\N_0$,  $b_1,\ldots,b_d\in\R$ s.t. for a.e. $\alpha\in\Tor$ there is $\vf_\alpha\in C(\Tor)$ s.t.
$$
f(x)=\sum_{i=1}^d b_i h(x+\beta_i)+\int_{\Tor} f(t)dt+\vf_\alpha(x)-\vf_\alpha(x+\alpha)\ \ (x\neq \beta_1,\ldots,\beta_d).
$$
\end{prop}

The following proposition was proved in \cite{Dolgopyat-Sarig-JSP}.
Let $(\Omega,\cB,\mu)$ be a probability space, and let $T:\Omega\to\Omega$ be a probability preserving map.
\begin{prop}
 Suppose $f=g+\vf-\vf\circ T$ $\mu$-a.e. with $f,g,\vf:\Omega\to\R$ measurable.
If the  ergodic sums of $g$ satisfy a TDLT along the orbit of a.e. $x$, then so do the ergodic sums of $f$.
\end{prop}

These results show that if Theorem \ref{ThBeckConverse} holds for every $f\in\cG$, then  Theorem \ref{ThBeckConverse} holds for any {\em discontinuous}  piecewise smooth function with zero mean.
As for {\em continuous} piecewise smooth functions with zero mean, these are $R_\alpha$-cohomologous to $g\equiv 0$ for a.e. $\alpha$ because the $b_i$ in Proposition \ref{PrPerNeg} must all vanish. Since the zero function does not satisfy the TDLT,
continuous piecewise smooth functions do not satisfy a TDLT.

%

%

\section{The set $\cA$ has full measure}
\subsection{Statement and plan of proof}\label{Section-cfe}
Let $\alpha$ be an irrational number, with continued fraction expansion $[a_0;a_1,a_2,a_3,\ldots]:=a_0+\cfrac{1}{a_1+\cdots}$, $a_0\in\Z$, $a_i\in\N$ ($i\geq 1$). We call $a_n$ the {\em quotients} of $\alpha$.  Let
$p_n/q_n$ denote the {\em principal convergents} of $\alpha$, determined recursively by $$q_{n+1}=a_{n+1}q_n+q_{n-1}, \quad p_{n+1}=a_{n+1} p_n+p_{n-1}$$ and  $p_0=a_0, q_0=1$; $p_1=1+a_1 a_0$, $q_1=a_1$. We call $q_n$ the {\em principal denominators} and $a_i$ the {\em partial quotients} of $\alpha$.
Sometimes -- but not always! -- we will write $q_k=q_k(\alpha)$, $p_k=p_k(\alpha)$, $a_k=a_k(\alpha)$.

Given $\cN\subset \N$ and  $M\geq 1$,
let $\cA=\cA(\cN, M)\subset (0,1)$ denote the set of irrational $\alpha\in (0,1)$ s.t. for some subsequence $n_k\uparrow\infty$,
\begin{equation}
\label{OddDenom}
\exists r_k \leq M\text{ s.t. } r_k q_{n_k}\in \cN \ , \
\dfrac{a_{n_k+1}}{(a_0+\cdots + a_{n_k})}\xrightarrow[k\to\infty]{}\infty.
\end{equation}
The {\em lower density} of $\cN$ is $d(\cN):=\liminf\frac{1}{N}\Card(\cN\cap[1,N])$.
The purpose of this section is to prove:

\begin{thm}\label{LmLargePC}
If a set $\cN$ has positive lower density then there exists $M$ such that
$\cA(\cN, M)$ has full Lebesgue measure in $(0,1)$.
\end{thm}

The proof consists of the following three lemmas:

\begin{lem}
\label{LmDV}
For almost all $\alpha$ there is $n_0=n_0(\alpha)$ s.t. if
$k\geq n_0$ and $ a_{k+1}>\frac{1}{4} k (\ln k) (\ln \ln k)$,
then
$ {a_{k+1}}/{(a_1+\dots+a_k)}\geq \frac{1}{8}\ln\ln k. $
\end{lem}

\begin{lem}
\label{LmOneApp}
Suppose $\alpha\in (0,1)\setminus\Q$ and $(p,q)\in\N_0\times\N$ satisfy
 $\gcd(p,q)= 1$  and
$\left|q\alpha-p\right|\leq \displaystyle\frac{1}{qL}$ where $L\geq 4$. Then
there exists  $k$ s.t. $q=q_k(\alpha)$ and
 $a_{k+1}(\alpha)\geq  \frac{1}{2} L
$.
\end{lem}


\begin{lem}
\label{LmGA}
Suppose $\psi:\R_+\to \R$ is a non-decreasing function s.t.
\begin{equation}
\label{KhCond1}
 \sum_n \frac{1}{n\psi(n)}=\infty .
\end{equation}
Suppose  $\cN\subset\N$ has positive lower density. For all  $M$ sufficiently large, for a.e. $\alpha\in (0,1)$ there are infinitely many pairs $(m,n)\in\N_0\times\N$
s.t.  $n\in \cN, \gcd(m,n)\leq M$, and
 $
\left|n\alpha-m\right|\leq \displaystyle\frac{1}{n\psi(n)}.
$
\end{lem}

\noindent
\medskip
{\em Remark 1.\/}
By the monotonicity of $\psi$, if $e^{k-1}<n<e^k$ then
$\psi\left(e^{k-1}\right)\leq \psi(n)\leq \psi\left(e^k\right).$ Hence \eqref{KhCond1} holds iff
$ \sum \frac{1}{\psi(e^k)}=\infty.$

\medskip
\noindent
{\em Remark 2.\/}
If $\cN=\N$, then Lemma \ref{LmGA} holds with
$M=1$ by the classical Khinchine Theorem. We do not know  if Lemma \ref{LmGA} holds with $M=1$ for any set $\cN$  with positive lower density.

\begin{proof}[Proof of Theorem \ref{LmLargePC} given Lemmas  \ref{LmDV}--\ref{LmGA}]
 We apply these lemmas with  $\psi(t)=c(\ln t) \; (\ln \ln t) \; (\ln \ln \ln t)$ and  $c>1/\ln(\frac{1+\sqrt{5}}{2})$.

 Fix $M>1$ as in  Lemma \ref{LmGA}.
Then $\exists \Omega\subset (0,1)$ of full measure s.t. for every $\alpha\in \Omega$ there  are infinitely many   $(m,n)\in\N_0\times\N$ as follows. Let
$m^\ast:=m/\gcd(m,n), \quad n^\ast:=n/\gcd(m,n), \quad p:=\gcd(m,n)$, then
\begin{enumerate}[(1)]
\item $p n^\ast\in\cN$,  $p\leq M$,
$|n^\ast\alpha-m^\ast|=\frac{|n\alpha-m|}{p}\leq \frac{1}{n^\ast \psi(n^\ast)}$
($\because n^\ast\leq n$);
\item  $\exists k$ s.t. $n^\ast=q_k(\alpha)$ and
$a_{k+1}(\alpha)\geq \frac{1}{2}\psi(q_k)$ ($\because$ Lemma \ref{LmOneApp}). By its recursive definition, $q_k\geq $ $k$-th Fibonacci number $\geq \frac{1}{3}(\frac{1+\sqrt{5}}{3})^k$. So for all $k$ large enough,
$a_{k+1}(\alpha)\geq \frac{1}{2}\psi(q_k)> \frac{1}{4}k(\ln k)(\ln\ln k)$;
\item $a_{k+1}/(a_1+\cdots+a_k)\geq  \frac{1}{8}\ln\ln k\to\infty$ ($\because$ Lemma \ref{LmDV}).
\end{enumerate}
So  every $\alpha\in\Omega$ belongs to $\cA=\cA(\cN,M)$, and $\cA$ has full measure.
\end{proof}

Next we prove Lemmas  \ref{LmDV}--\ref{LmGA}.

\subsection{Proof of Lemma \ref{LmDV}}
By \cite{Diamond-Vaaler},   for almost every $\alpha$
$$ \frac{\left(a_1+\dots+a_{k+1}\right)-\max_{j\leq k+1} a_j}{k\ln k}\to \frac{1}{\ln 2}<2 . $$
So  if $k$ is large enough,  and $a_{k+1}>\frac{1}{4} k(\ln k)(\ln\ln k)$ then \\
$$ \max_{j\leq k+1} a_j=a_{k+1}, \quad
\frac{a_1+\dots+a_k}{k\ln k}\leq 2, \text{ and }
\frac{a_{k+1}}{a_1+\cdots+a_k}>\frac{1}{8}\ln\ln k.
\Box$$

\subsection{Proof of Lemma \ref{LmOneApp}}
$\text{}$
For every $(p,q)$ as in the lemma,
$
|q\alpha-p|<\frac{1}{2q}.
$
A classical result in the theory of continued fractions \cite[Thm 19]{Khintchine-Continued-Fractions} says that in this case $\exists k$ s.t. $q=q_k(\alpha), p=p_k(\alpha)$.

To estimate $a_{k+1}=a_{k+1}(\alpha)$ we recall the following facts, valid for the principal denominators of any irrational $\alpha\in (0,1)$ \cite{Khintchine-Continued-Fractions}:
\begin{enumerate}[(a)]
\item $|q_k\alpha-p_k|>\frac{1}{q_k+q_{k+1}}$;
\item $q_{k+1}+q_k<(a_{k+1}+2)q_k$, whence by (a) $a_{k+1}>\frac{1}{q_k|q_k\alpha-p_k|}-2$.
\end{enumerate}
In our case, $|q_k\alpha-p_k|=|q\alpha-p|\leq \frac{1}{q_k L}$,
so
$a_{k+1}>L-2\geq \frac{L}{2}$.
\hfill$\Box$

\subsection{Preparations for the proof of Lemma \ref{LmGA}}
 Let $(\Omega,\cF,\Prob)$ be a probability space, and $A_k\in\cF$ be measurable events. Given $D>1$, we say that $A_k$ are {\em $D$-quasi-independent}, if
\begin{equation}
\label{QI}
\Prob(A_{k_1}\cap A_{k_2})\leq D\Prob(A_{k_1})\Prob(A_{k_2})\text{ for all }k_1\neq k_2.
\end{equation}
The following proposition is  a slight variation on
Sullivan's Borel--Cantelli Lemma from (\cite{Sullivan-Logarithm-Law}):
\begin{prop}
\label{LmBCSul}
For every  $D\geq 1$
there exists a constant $\delta(D)>0$ such that the following holds in any probability space:
\begin{enumerate}[(a)]
\item If $A_k$ are $D$-quasi-independent measurable events s.t. $\DS \lim_{k\to\infty} \Prob(A_k)=0$ but
$\sum_k \Prob(A_k)=\infty$, then $ \Prob(A_k \text{ occurs infinitely often})\geq \delta(D)$.
\item The quasi-independence assumption in (a) can be weakened to the assumption that
for some $r\in\N$,
$\Prob(A_{k_1}\cap A_{k_2})\leq D\Prob(A_{k_1})\Prob(A_{k_2})$ for all $|k_2-k_1|\geq r.$
%
%
\item One can take $\delta(D)=\frac{1}{2D}$.
\end{enumerate}
\end{prop}
%

\begin{proof}
Since $\Prob(A_k)\to 0$ but $\sum \Prob(A_k)=\infty$, there is an increasing
sequence $N_j$ such that
$\displaystyle\lim_{j\to\infty} \sum_{k=N_j+1}^{N_{j+1}} \Prob(A_k)=\frac{1}{D}. $

Let $B_j$ be the event that at least one of events $\{A_k\}_{k=N_j+1}^{N_{j+1}}$
occurs. Since
$B_j=\biguplus_{k=N_{j+1}}^{N_{j}+1}\bigl(A_k\setminus\bigcup_{j=N_j+1}^{k-1} A_j\bigr)$,
\begin{align*}
\Prob(B_j)&\geq \sum_{k=N_j+1}^{N_{j+1}} \Prob(A_k)-
\sum_{N_j+1\leq k_1<k_2\leq N_{j+1}} \Prob(A_{k_1}\cap A_{k_2})\\
& \geq \sum_{k=N_j+1}^{N_{j+1}} \Prob(A_k)-
D\!\!\!\!\!\sum_{N_j+1\leq k_1<k_2\leq N_{j+1}}\!\!\!\!\! \Prob(A_{k_1})\Prob(A_{k_2})\\
& \geq \sum_{k=N_j+1}^{N_{j+1}} \Prob(A_k)-\frac{D}{2} \left(\sum_{k=N_j+1}^{N_{j+1}} \Prob(A_k)\right)^2.
\end{align*}
Since
{ $\DS \lim_{j\to\infty} \sum_{k=N_j+1}^{N_{j+1}} \Prob(A_k)=\frac{1}{D}$}
 and $D\geq 1,$
 { $\liminf \Prob(B_j)\geq \frac{1}{2D}$.}

Let $E$ denote the  event that $A_j$ happens infinitely often. $E$  is also the event that $B_j$ happens infinitely often, therefore  $E=\bigcap_{n=1}^\infty \bigcup_{j=n+1}^\infty B_j$. In a probability space, the measure of a decreasing intersection of sets is the limit of the measure of these sets. So $\Prob(E)\geq \liminf \Prob(B_j)\geq \frac{1}{2D}$, proving (a) and (c).

Part (b) follows from  part (a) by applying it to the sets $\{A_{kr+\ell}\}$ where $0\leq \ell\leq r-1$ is chosen to get
$\DS \sum_k \Prob(A_{kr+\ell})=\infty.$
\end{proof}
%
%
%
%
%
%
%
%
%

The {\em multiplicity} of a collection of measurable sets  $\{E_k\}$ is defined to be the largest $K$ s.t.
there are $K$ different $k_i$ with $\Prob(\bigcap_{i=1}^K E_{k_i})>0$.
\begin{prop}\label{LmFinMult}
Let $E_k$ be measurable sets in a finite measure space. If the multiplicity of $\{E_k\}$ is less than $K$,  then
$$ \mes\left(\bigcup_k E_k\right)\geq \frac{1}{K} \sum_k \mes(E_k). $$
\end{prop}
\begin{proof}
$1_{\bigcup E_i}\geq \frac{1}{K}\sum 1_{E_i}$ almost everywhere.
\end{proof}

\begin{prop}\label{Sylvester} For every non-empty open interval $I\subset [0,1]$,\\
$\Card\{(m,n)\in\{0,\ldots,N\}^2:\frac{m}{n}\in I\ , \ \gcd(m,n)=1\}\sim 3\mes(I)N^2/\pi^2$, as $N\to\infty$.
\end{prop}
\begin{proof}
This  classical fact due to  Dirichlet follows from the inclusion-exclusion principle and the identity
$\zeta(2)=\pi^2/6$,  see \cite[Theorem 459]{HardyWright}.
\end{proof}

\begin{prop}\label{MultiStepCF}
Suppose $\alpha=[0;a_1,a_2,\ldots]$ and $\ov{\alpha}=[0;a_{\ell+1},a_{\ell+2},\ldots]$. Then the principal convergents $\ov{p}_{\ov{\ell}}/\ov{q}_{\ov{\ell}}$ of $\ov{\alpha}$ and the principal convergents $p_\ell/q_\ell$
of $\alpha$ are related by
 $\left(\begin{array}{cc} p_{l+\brl} & p_{l+\brl+1} \\  q_{l+\brl} & q_{l+\brl+1} \end{array}\right)  =
 \left(\begin{array}{cc} p_{l-1} & p_{l} \\  q_{l-1} & q_{l} \end{array}\right)
 \left(\begin{array}{cc} \bar{p}_{\brl} & \bar{p}_{\brl+1} \\  \bar{q}_{\brl} & \bar{q}_{\brl+1} \end{array}\right)
$
\end{prop}
\begin{proof}
Since $a_0=0$,
the recurrence relations for $p_n/q_n$ imply
$$\left({
\begin{array}{ll} p_{n} & p_{n+1}\\ q_{n} & q_{n+1}
\end{array}
}\right)=\left({
\begin{array}{ll} p_{n-1} & p_{n}\\ q_{n-1} & q_{n}
\end{array}
}\right) \left(
\begin{array}{lc} 0 & 1\\ 1 & a_{n+1}
\end{array}
\right)\ , \
\left(\begin{array}{ll} p_{0} & p_{1}\\ q_{0} & q_{1}
\end{array}
\right)
=\left(\begin{array}{ll} 0 & 1\\ 1 & a_1
\end{array}
\right).
$$
So
$\left({\small \begin{array}{ll} p_{n} & p_{n+1}\\ q_{n} & q_{n+1}
\end{array}
}\right)=\left({\small \begin{array}{ll} 0 & 1\\ 1 &  a_{1}
\end{array}
}\right)\cdot\ldots\cdot \left({\small \begin{array}{lc} 0 & 1\\ 1 & a_{n+1}
\end{array}
}\right)
$. It follows that
$
\left(\begin{array}{cc} p_{l+\brl} & p_{l+\brl+1} \\  q_{l+\brl} & q_{l+\brl+1} \end{array}\right)  =
\left(\begin{array}{cc} p_{l-1} & p_{l} \\  q_{l-1} & q_{l} \end{array}\right)
 \left(\begin{array}{cc} \bar{p}_{\brl} & \bar{p}_{\brl+1} \\  \bar{q}_{\brl} & \bar{q}_{\brl+1} \end{array}\right),
$
where $\brp_i/\brq_i$ are the principal convergents of $\bar{\alpha}:=[0;a_{l+1},a_{l+2},\ldots]$.
\end{proof}

%
%

\subsection{Proof of Lemma \ref{LmGA}}
Without loss of generality,
$\displaystyle \lim_{t\to\infty} \psi(t)=\infty$, otherwise replace $\psi(t)$ by
the bigger monotone function $\psi(t)+\ln t$.

Fix $M>1$, to be determined later. Let
\begin{align*}
&\Omega_k:=\{(m,n)\in\N^2: n\in\cN,n\in [e^{k-1},e^k],0<m<n, \gcd(m,n)\leq M\},\\
&A_{m,n,k}:=\{\alpha\in\T: |n\alpha-m|\leq \tfrac{1}{ e^k \psi(e^k)}\}, \\
&\cA_k:=\bigcup_{(m,n)\in\Omega_k} A_{m,n,k},\\
&\cA:=\{\alpha\in\T: \alpha\text{ belongs to infinitely many }\cA_k\}.
\end{align*}
The lemma is equivalent to saying that $\cA$  has full Lebesgue measure for a suitable choice of $M$.

We will prove a slightly different statement. Fix $\eps>0$ small.  Given an non-empty interval $I\subset [\eps,1-\eps]$,
let
\begin{align*}
&\Omega_k(I):=\{(m,n)\in\Omega_k: \frac{m}{n}\in I\}\\
&\cA_k(I):=\bigcup_{(m,n)\in\Omega_k(I)}A_{m,n,k}\\
&\cA(I):=\{\alpha\in\T: \alpha\text{ belongs to infinitely many }\cA_k(I)\}.
\end{align*}
We will prove that there exists a positive constant $\delta=\delta(\eps,M)$ s.t. for all intervals $I\subset [\eps,1-\eps]$,  $\mes(\cA(I)\cap I)\geq \delta\mes(I)$. It then follows by a standard density point argument (see below) that $\cA\cap [\eps,1-\eps]$ has full measure. Since $\eps$ is arbitrary,  the lemma is proved.

\medskip
\noindent
{\sc Claim 1.} There exist $K=K(\eps)$ s.t. for every $k> K$, the multiplicity of $\{A_{m,n,k}\}_{(m,n)\in\Omega_k(I)}$ is uniformly bounded by $M$.

\medskip
\noindent
{\sc Proof:} Suppose $(m_i,n_i)\in\Omega_k(I)$ and
$A_{m_1, n_1, k}\cap A_{m_2, n_2, k}\neq \emptyset$. Then  there is $\alpha$ s.t.
$|n_i\alpha-m_i|\leq \delta_k:=\frac{1}{ e^k \psi(e^k)}$.
Choose $K=K(\epsilon)$ so large that $k>K\Rightarrow \delta_k<\frac{\eps}{4e^k}$.

If $k>K$, then
$\alpha\geq \frac{m_i}{n_i}-\delta_k> \min I-\frac{\eps}{2}>\frac{\eps}{2}$.
Let $r_i:=\gcd(m_i,n_i)$ and $(n_i^\ast,m_i^\ast):=\frac{1}{r_i}(n_i,m_i)$. Then $|n_i^\ast\alpha-m_i^\ast|\leq \delta_k$ and  $m_i^\ast\leq n_i^\ast\leq n_i\leq e^k$, so
$
|n_2^\ast m_1^\ast-n_1^\ast m_2^\ast|=\frac{1}{\alpha}|m_1^\ast(n_2^\ast\alpha-m_2^\ast)- m_2^\ast(n_1^\ast \alpha-m_1^\ast)|\leq \frac{2e^k \delta_k}{\eps/2}<1.
$
So $n_2^\ast m_1^\ast=n_1^\ast m_2^\ast$. Since $\gcd(n_i^\ast,m_i^\ast)=1$, $(n_1^\ast,m_1^\ast)=(n_2^\ast,m_2^\ast)$. It follows that
$
(n_2,m_2)\in\{(rn_1^\ast,rm_1^\ast):r=1,\ldots,M\}
$.
So the multiplicity of $\{A_{m,n,k}\}_{(m,n)\in\Omega_k(I)}$ is uniformly bounded by $M$.

\medskip
\noindent
{\sc Claim 2.} Let $d(\cN):=\liminf \frac{1}{N}\Card(\cN\cap [1,N])>0$, then  there exists $M=M(\cN)$ and  $\tK=\tK(\eps,\cN, |I|)$  s.t. for all $k>\tK$,
\begin{equation}\label{mes[A(I)]}
 \frac{d(\cN)\mes(I)}{4M{\psi(e^k)}}\leq \mes(\cA_k(I))\leq \frac{6\mes(I)}{ \psi(e^k)}.
\end{equation}
In particular, $\mes(\cA_k(I))\xrightarrow[k\to\infty]{}0$ and $\sum \mes(\cA_k(I))=\infty$.

\medskip
\noindent
{\sc Proof:}
$ \mes(A_{m,n,k} )=\mes\left(\left[\frac{m}{n}-r_{m,n}, \frac{m}{n}+r_{m,n}\right]\right)=
2 r_{m,n} $
where
$ r_{m,n}={ \dfrac{1}{n e^k \psi(e^k)}}$. Since $n\in [e^{k-1},e^k]$,
\begin{equation}\label{mes(A_k(I))}
 \frac{\Card(\Omega_k(I))}{M { e^{2k} \psi(e^k)}}\leq \mes\bigl(\cA_k(I)\bigr)\leq
 \frac{{ e}\;\Card(\Omega_k(I))}{{e^{2k} \psi(e^k)}},
\end{equation}
where the lower bound uses Claim 1 and Proposition \ref{LmFinMult}.

$\Card(\Omega_k(I))$ satisfies the bounds  $A-B\leq \Card(\Omega_k(I))\leq A$ where
\begin{align*}
A&:=\Card\{(m,n): n\in \cN,\; n\in [e^{k-1}, e^{k}],\; \frac{m}{n}\in I\}\\
B&:=\Card\{(m,n): n\in \cN,\; n\in [e^{k-1}, e^{k}],\; \frac{m}{n}\in I, \gcd(m,n)\geq M\}.
\end{align*}

Choose $\tK=\tK(\eps,\cN,  |I|)>K(\eps)$ s.t. for all
$k>\tK$
\begin{enumerate}[(a)]
\item $
\Card\{n\in\cN: 0\leq n\leq e^k\}\geq \frac{1}{\sqrt{2}}d(\cN)
$
\item    $\Card\{n\in [e^{k-1},e^k]\cap\N:p|n\}\leq 2(e^k-e^{k-1})/p$ for all
$p\geq 1$;
\item  For all $n>e^{\tK-1}$, $p\geq 1$,  
$$
\frac{n}{p\sqrt{2}}\mes(I)\leq\Card\{m\in \N: \frac{m}{n}\in I, p|m\}\leq \frac{2n}{p}\mes(I).
$$
\end{enumerate}
If $k>\tK$, then $\frac{1}{2}d(\cN) e^{2k}\mes(I)\leq A\leq 2 e^{2k}\mes(I)$
and
\begin{align*}
B&\leq \sum_{p=M}^\infty \Card\{(m,n): n\in [e^{k-1}, e^k],\;\frac{m}{n}\in I, p|m, p|n \}
\end{align*}
\begin{align*}
&\leq \sum_{p=M}^\infty \frac{2(e^k-e^{k-1})}{p}\cdot \frac{2e^k\mes(I)}{p}<4e^{2k}\mes(I)\sum_{p=M}^\infty\frac{1}{p^2}\\
&\leq \frac{1}{4}d(\cN)e^{2k}\mes(I),\text{ provided  we choose $M$ s.t. $\sum_{p=M}^\infty p^{-2}<\frac{1}{16}d(\cN)$.}
\end{align*}
Together we get $\frac{1}{4} d(\cN) e^{2k}\mes(I)\leq \Card(\Omega_k(I))\leq 2 e^{2k}\mes(I)$.
The claim now follows from \eqref{mes(A_k(I))}.

\medskip
\noindent
{\sc Claim 3.} There exists $D=D(\cN,M)$, $r=r(M)$, and $\hK=\hK(\eps,\cN, I)$ s.t. for all $k_1,k_2>\hK$ s.t. $|k_1-k_2|>r(M)$,
\begin{equation}\label{Quasi-Independence}
\mes(\cA_{k_1}(I)\cap\cA_{k_2}(I)| I)\leq D\mes(\cA_{k_1}(I)| I)\mes(\cA_{k_2}(I)|I)
\end{equation}

\medskip
\noindent
{\sc Proof:} By Claim 2, if $k_1, k_2$ are large enough,  then
\begin{equation}\label{RHS-below}
\mes(\cA_{k_1}(I)| I)\mes(\cA_{k_2}(I)| I)
\geq \left(\frac{d(\cN)}{5M}\right)^2
{\frac{1}{\psi(e^{k_1})\psi(e^{k_2})},}
\end{equation}
where we put $5$ instead of $4$ in the denominator to deal with edge
effects arising from $\mes(\cA_k(I)\setminus I)=O\left(\frac{1}{e^k \psi(e^k)}\right)$.

To prove the claim, it remains to bound
$\displaystyle \mes\left(\cA_{k_1}(I) \cap \cA_{k_2}(I)\Big| I\right)$ from above by
$\frac{\text{const}}{R_1 R_2}$, where $R_i:={\psi(e^{k_i})}$.

A {\em cylinder} is a set of the form
$$
[\![a_1,\ldots,a_n]\!]=\{\alpha\in (0,1)\setminus\Q: a_i(\alpha)=a_i\ \ (1\leq i\leq n)\}.
$$
Equivalently, $\alpha\in [\![a_1,\ldots,a_n]\!]$ iff $\alpha$ has an infinite continued fraction expansion of the form $\alpha=[0;a_1,\ldots,a_n,\ast,\ast,\ldots]$.

Our plan is to cover $\cA_{k_i}(I)$  by unions of {cylinders}  of total measure $O(1/R_i)$,
and then use the following well-known  fact: There is a constant $G>1$ s.t. for any $(a_1,\ldots,a_n,b_1,\ldots,b_m)\in\N^{n+m}$,
\begin{equation}\label{Gibbs}
G^{-1}\leq \frac{\mes[\![a_1,\ldots,a_n;b_1,\ldots,b_m]\!]}{\mes[\![a_1,\ldots,a_n]\!]\mes[\![b_1,\ldots,b_m]\!]}\leq G.
\end{equation}
This is because the invariant measure $\frac{1}{\ln 2}\frac{dx}{1+x}$ of $T:(0,1)\to (0,1)$, $T(x)=\{\frac{1}{x}\}$ (the Gauss map) is a Gibbs-Markov measure, because of the bounded distortion of $T$, see \S2 in \cite{Aaronson-Denker-Urbanski}.

To  cover $\cA_{k_i}(I)$ by cylinders, it is enough to cover $A_{m,n,k_i}$ by cylinders for every $(m,n)\in\Omega_{k_i}(I)$. Suppose $\alpha\in A_{m,n,k_i}$.   Then $r:=\gcd(m,n)\leq M$ and $(m^\ast,n^\ast):=\frac{1}{r}(m,n)$ satisfies
$$
\gcd(m^\ast,n^\ast)=1,\ n^\ast\in \bigcup_{|k_i^\ast-k_i|\leq \ln M}[e^{k_i^\ast-1},e^{k_i^\ast}],\
|n^\ast\alpha-m^\ast|< \frac{1}{n^\ast R_i}.
$$

Assume $k_i$ is so large that $R_i=\psi(e^{k_i})\geq 4.$ Then Lemma \ref{LmOneApp} gives
$a_{l_i+1}>\frac{R_i}{2}.$
Thus $\cA_{k_i}(I)\subset \cC_{k_i}(I, R_i)$ where
$$
\cC_{k}(I,R):=\bigcup_{ k^\ast\in [k-\ln M, k]} \left\{\alpha\in (0,1)\setminus\Q: \exists\ell\text{ s.t. }
\begin{array}{l}
q_\ell(\alpha)\in [e^{k^\ast-1},e^{k^\ast}],\\
a_{{\ell}+1}
(\alpha)\geq R/2\\
p_\ell(\alpha)/q_\ell(\alpha)\in I
\end{array}\right\}.
$$
This is a union of cylinders, because $q_\ell(\alpha), p_\ell(\alpha), a_{\ell+1}(\alpha)$ are constant on cylinders of length $\ell+1$.

We claim that for some $c^\ast(M)$ which only depends on $M$,
for all $k_i$ large enough,
\begin{equation}
\label{MesCkR}
\mes(\cC_{k_i}(I,R_i))\leq \frac{c^*(M) \mes(I)}{R_i}.
\end{equation}
Every rational $\frac{m}{n}\in (0,1)$ has two finite continued fraction expansions: $[0;a_1,\ldots,a_{\ell}]$ and $[0;a_1,\ldots,a_{\ell}-1,1]$ with $a_{\ell}>1$. We write $\ell=\ell(\frac{m}{n})$ and  $a_i=a_i(\frac{m}{n})$. With this notation
\begin{align*}
&\cC_{k_i}(I,R_i)=\underset{n\in [e^{k_i^\ast-1},e^{k_i^\ast}]}{\bigcup_{k_i^\ast\in [k_i-\ln M, k_i]}}\;\underset{m/n\in I}{\bigcup_{\gcd(m,n)=1}}\; \bigcup_{ b>R_i/2} [\![a_1(\tfrac{m}{n}),\ldots,a_{\ell(\frac{m}{n})}(\tfrac{m}{n}),b]\!]\\
&\hspace{6.5cm}\cup [\![a_1(\tfrac{m}{n}),\ldots,a_{\ell(\frac{m}{n})}(\tfrac{m}{n})-1,1,b]\!].
\end{align*}
We have $[\![a_1,\ldots,a_\ell]\!]=(\frac{p_\ell+p_{\ell-1}}{q_\ell+q_{\ell-1}},\frac{p_\ell}{q_\ell})$ or  $(\frac{p_\ell}{q_\ell},\frac{p_\ell+p_{\ell-1}}{q_\ell+q_{\ell-1}})$, depending on the parity of $\ell$  \cite{Khintchine-Continued-Fractions}. Since {
$|p_\ell q_{\ell-1}-p_{\ell-1}q_\ell|=1$} and
{ $q_{\ell+1}=a_{\ell+1}q_\ell+q_{\ell-1}$,}
 we have
$\mes([\![a_1,\ldots,a_{\ell},b]\!])=\tfrac{1}{q_{\ell+1}(q_{\ell+1}+q_\ell)}=\tfrac{1}{(bq_{\ell}+q_{\ell-1})((b+1)q_\ell+q_{\ell-1})}\leq \tfrac{1}{b(b+1) q_\ell^2},$
leading to
\begin{align*}
&\mes(\cC_{k_i}(I,R_i))\leq
\sum_{k_i^\ast\in [k_i-\ln M, k_i]}\sum_{n\in [e^{k_i^\ast-1},e^{k_i^\ast}]}\underset{m/n\in I}{\sum_{\gcd(m,n)=1}}\sum_{b>R_i/2}\frac{2}{n^2 b(b+1)}\\
&\leq \frac{  8\ln M}{e^{2(k_i-1-\ln M)} R_i}\sum_{n= 1}^{e^{k_i}M}\#\{m\in\N: \frac{m}{n}\in I, \gcd(m,n)=1\}
\leq \frac{c^\ast(M)}{R_i}\mes(I)
\end{align*}
where $c^\ast(M)$ only depends on $M$. The last step uses Prop. \ref{Sylvester}.

Next we cover $\cA_{k_1}(I)\cap\cA_{k_2}(I)$ by cylinders. Suppose without loss of generality that $k_2>k_1$.
Arguing as before one sees that  if
\begin{equation}
\label{KAstApart}
k_2>k_1+\ln M+1,
\end{equation}
then $\cA_{k_1}(I)\cap\cA_{k_2}(I)$ can be covered by sets
$
[\![a_1,\ldots,a_\ell,b,\bra_1,\ldots,\bra_{\brl},\brb]\!]
$ as follows:
The  convergents $p_i/q_i$ of (every) $\alpha$ in $
[\![a_1,\ldots,a_\ell,b,\bra_1,\ldots,\bra_{\brl},\brb]\!]
$, ($1\leq i\leq l+\brl+2$), satisfy
\begin{enumerate}[(a)]
\item $q_l\in [e^{k_1^\ast-1},e^{k_1^\ast}]$,
$ k_1^\ast\in [k_1-\ln M, k_1]$, $p_l/q_l\in I$,
$b\geq {R_1/2}$;
\item $q_{l+\brl+1}\in [e^{k_2^\ast-1},e^{k_2^\ast}]$,
$k_2^\ast\in [k_2-\ln M, k_2]$, $p_{\brl}/q_{\brl}\in I$,
$\brb\geq {R_2/2}$
\item $k_2^\ast>k_1^\ast$ (this is where \eqref{KAstApart} is used).
\end{enumerate}

We claim that
\begin{equation}
\label{FirstCyl}
[\![a_1,\ldots,a_\ell,b]\!]\subset\cC_{k_1}(I),
\end{equation}
\begin{equation}
\label{BSize}
b\leq e^{k_2^\ast-k_1^\ast+1},
\end{equation}
\begin{equation}
\label{SecondCyl}
 [\![\ov{a}_1,\ldots,\ov{a}_{\ov{\ell}},\ov{b}]\!]\subset \bigcup_{|r|\leq 3}
\cC_{k_2-k_1+r-\ln b}([0,1],R_2).
\end{equation}
 \eqref{FirstCyl} follows from (a).
Next, $e^{k_2^\ast}\geq q_{l+1}\geq b q_l\geq b e^{k_1^\ast-1} $
proving \eqref{BSize}.
To prove  \eqref{SecondCyl},
let $\brp_i/\ov{q}_i$, $1\leq i\leq \ov{\ell}+2$, be the
principal convergents of (every)
$\ov{\alpha}\in [\![b, \ov{a}_1,\ldots,\ov{a}_{\ov{\ell}},\ov{b}]\!]$.
By Prop.~\ref{MultiStepCF}, $q_{l+1+\brl}=q_{l-1}\brp_{\brl+1}+q_l\brq_{\brl+1}$, whence
$q_l\brq_{\brl+1}\leq q_{l+1+\brl}\leq 2q_l\brq_{\brl+1}$. Since $q_l\in [e^{k_1^\ast-1},e^{k_1^\ast}]$ and $q_{l+\brl+1}\in [e^{k_2^\ast-1},e^{k_2^\ast}]$,
\begin{equation}\label{Gibbs-prime}
e^{k_2^\ast-k_1^\ast-2}\leq \frac{q_{l+1+\brl}}{2q_l}\leq \brq_{\brl+1}\leq \frac{q_{l+1+\brl}}{q_l}\leq e^{k_2^\ast-k_1^\ast+1}.
\end{equation}
Next, let $\wt{p}_i/\wt{q}_i$ ($1\leq i\leq \brl$) denote the principal convergents of (every) $\wt{\alpha}\in [\![ \ov{a}_1,\ldots,\ov{a}_{\ov{\ell}},\ov{b}]\!]$. Then $\frac{\brp_{\brl+1}}{\brq_{\brl+1}}=1/(b+\frac{\wt{p}_{\brl}}{\wt{q}_{\brl}})$, so $\brq_{\brl+1}=b\wt{q}_{\brl}+\wt{p}_{\brl}$, whence
$b\wt{q}_{\brl}\leq \brq_{\brl+1}\leq (b+1)\wt{q}_{\brl}$. Thus $\wt{q}_{\brl}\in [(b+1)^{-1}\brq_{\brl+1},b^{-1}\brq_{\brl+1}]$. It follows that the $\brl$-th principal convergent of every $\wt{\alpha}\in [\![ \ov{a}_1,\ldots,\ov{a}_{\ov{\ell}},\ov{b}]\!]$ satisfies
\begin{equation}\label{Gibbs-Double-Prime}
\wt{q}_{\brl}\in [e^{k_2^\ast-k_1^\ast-3-\ln b},e^{k_2^\ast-k_1^\ast+1-\ln b}].
\end{equation}
It is now easy to see \eqref{SecondCyl}.

By \eqref{SecondCyl},
$\displaystyle
\cA_{k_1}(I)\cap\cA_{k_2}(I)\cap I\subset \bigcup_{|r|\leq 3}\biguplus_{[\un{a},b]\subset \cC_{k_1}(I)}{\biguplus_{[\un{a}',b']\subset \cC_{k_2-k_1+r-{ \ln b}}([0,1])}}\!\!\!\!\!\!\!\!\!\!\!\!\!\!\!\!\!\![\![\un{a},b,\un{a}',{b}]\!].
$
 Now arguing as in the proof of \eqref{MesCkR} and using \eqref{Gibbs} we obtain
\begin{align}
&\mes(\cA_{k_1}(I)\cap\cA_{k_2}(I)\cap I)\leq \notag\\
&\leq
\!\!\!\! \sum_{\overset{{k_i^\ast\in [k_i-\ln M, k_i]}}{i=1,2 ; |r|\leq 3}}\;\;
\!\!\sum_{n\in [e^{k_1^\ast-1},e^{k_1^\ast}]}\;\;
\sum_{\underset{m/n\in I}{\gcd(m,n)=1}}
\sum_{b=[R_1/2]}^{[\exp(k_2^\ast-k_1^\ast+1)]}
\tfrac{2G\mes(\cC_{k_2-k_1+r-{\ln b}}([0,1], R_2)) }{n^2 b(b+1)}
\notag \\
&\leq \frac{\const\mes(I)}{R_1 R_2}.\label{kim-jong-un}
\end{align}

\eqref{kim-jong-un} uses the estimate $\mes(\cC_{k_2-k_1+r-{\ln b}}([0,1],{R_2}))=O(1/R_2)$ which is also valid when
$k_2-k_1+r-{\ln b}$ is small, provided we choose $M$ large enough so that the asymptotic in Prop. \ref{Sylvester} holds for all $N>M$ with $I=[0,1]$. See the proof of \eqref{MesCkR}.

Combining \eqref{kim-jong-un} with \eqref{RHS-below}, we find that
{ under \eqref{KAstApart}}
$\cA_{k_i}(I)$ are $D$-quasi-independent for sufficiently large $D$, proving Claim 3.

\medskip
Claims 2 and 3 allow us to apply Sullivan's Borel--Cantelli Lemma (Prop. \ref{LmBCSul}). We obtain $\delta=\delta(M)$ s.t. for every interval $I\subset [\eps,1-\eps]$, $\mes(\cA\cap I)\geq \delta\mes(I)$. This means that $[\eps,1-\eps]\setminus\cA$ has no Lebesgue density points, and therefore must have measure zero.
So $\cA$ has full measure in $[\eps,1-\eps]$. Since $\eps$ is arbitrary, $\cA$ has full measure. \hfill$\Box$

\section{Proof of Theorem \ref{ThBeckConverse}}
\label{ScResII}
As explained in Section \ref{ScPrel}, it is enough to prove Theorem \ref{ThBeckConverse} for
$ f(x):=\sum_{m=1}^d b_m h(x+\beta_m)\not\equiv 0$ with $h(x)=\{x\}-\frac{1}{2}$. Without loss of generality,  $\b_i$ are different and $b_i\neq 0$.
Notice that
$$ \displaystyle f(x)=-
\sum_{m=1}^d b_m\sum_{j=1}^\infty \frac{\sin(2\pi j (x+\beta_m))}{\pi j}.$$
Therefore  $\displaystyle \|f\|_2^2=\frac{1}{2\pi^2}\sum_n \frac{1}{ n^2} D(\beta_1 n,\ldots,\beta_d n)$,
where $D:\T^d\to \R$ is
\begin{align}\label{D-function}
D(\gamma_1, \dots, \gamma_d)&:=\int_0^1 \left[\sum_{m=1}^d b_m \sin(2\pi (y+\gamma_m))\right]^2 dy.
\end{align}
Since $f\not\equiv 0$, $D(\beta_1 n,\dots, \beta_d n)>0$ for some $n$.
Let $\fT$ denote the closure in $\T^d$ of
$\mathbb{O}:=\{(\beta_1 n,\dots, \beta_d n)  \mod\Z: n\in\Z\}.$
This is a minimal set for the translation by $(\beta_1,\ldots,\beta_d)$ on $\T^d$, so
a standard compactness argument shows that
the set
\begin{equation}
\label{RelBigFC}
\cN:=\{n\in \N: \; D(\beta_1 n,\ldots,\beta_d n)> \eps_0\}
\end{equation}
is syndetic: its gaps are bounded.
Thus  $\cN$  has positive lower density.

By Theorem \ref{LmLargePC}, if $M$ is sufficiently large then the set
$\cA:=\cA(\cN,M)$ has full measure in $\T.$ Let
$$
S_n(\alpha,x):=\sum_{k=0}^{n-1}f(x+k\alpha).
$$
The proof of Theorem \ref{ThBeckConverse} for $f(x)$ above consists of two parts:
\begin{thm}
\label{ThUniform}
Suppose $\alpha\in \cA$,  then
for a.e. $x\in [0,1)$, there exist
$A_k(x)\in \R$ and  $B_k(x), N_k(x)\to\infty$  such that
$$
\frac{S_n(\alpha,x)-A_k(x)}{B_k(x)}\xrightarrow[k\to\infty]{\text{dist}}\mathrm{U}[0,1],\text{ as }n\sim \mathrm{U}(0,\ldots,N_k(x)).
$$
\end{thm}

\begin{thm}
\label{ThNonUniform}
Suppose $\alpha\in \cA$,  then
for a.e. $x\in [0,1)$, there are no
$A_N(x)\in \R$ and  $B_N(x)\to\infty$  such that
$$
\frac{S_n(\alpha,x)-A_N(x)}{B_N(x)}\xrightarrow[N\to\infty]{\text{dist}}\mathrm{U}[0,1],\text{ as }n\sim \text{U}(0,\ldots,N).
$$
\end{thm}

\subsection{Preliminaries}

\begin{lem}\label{Lem-Sing-dist}
$S_{q}(\alpha,\cdot):\T\to\R$ has $d q$ discontinuities.
\end{lem}
\begin{proof}
The discontinuities of $S_q$ are preimages of discontinuities of $f$ by $R_\alpha^{-k}$
with $k=0,1,\ldots,q-1.$
\end{proof}

\begin{lem}
\label{LmLip}
Let $C:=\sup|f'|\leq |\sum b_m|$.
If $x', x''$ belong to same continuity component of $R_\alpha^r$ then
$$ |S_r(\alpha, x')-S_r(\alpha, x'')|\leq Cr |x'-x''|.$$
\end{lem}

\begin{proof}
Since
$|S_r'|=\left|\sum_{k=0}^{r-1} f'(x+k\alpha)\right|\leq Cr,$ the restriction of $S_r$ to
on each continuity component
is Lipshitz with Lipshitz constant $C r.$
\end{proof}

\begin{lem}\label{Lem-mu-n}
There are constants $C_1, C_2$ such that the following holds.
Suppose that $q_n$ is a principal denominator of $\alpha$, and $q_{n+1}>c q_n$ with $c>1$.
Let $\mu_n(x):=S_{q_n}(\alpha,x)$, then
\begin{equation}
\label{ChangeInc}
\mes\left\{x:\;S_{\ell q_n}(\alpha,x)=\ell\mu_n\pm C_1 \frac{\ell^2}{c} \text{ for }\ell=0,\ldots, k
\right\}
>1-C_2\frac{k}{c}.
\end{equation}
\end{lem}

\begin{proof}
If $x$ and  $x+\ell q_n \alpha$ belong to the same continuity interval of $R_\alpha^{q_n}$
for all $\ell=0,\ldots,k$
then we have by Lemma \ref{LmLip} { that for $\ell\leq k$}
\begin{align*}
&|S_{\ell q_n}(\alpha,x)-\ell\mu_n|\leq \sum_{j=0}^{\ell-1}|S_{q_n}(\alpha,x+ jq_n\alpha)-S_{q_n}(\alpha,x)|\leq  Cq_n \sum_{j=0}^{\ell-1}\|j q_n\alpha\|\\
&\leq \frac{Cq_n}{q_{n+1}}\sum_{j=0}^{\ell-1} j \leq \frac{C_1\ell^2}{c},\text{ where }C_1:=C/2.
\end{align*}
Therefore if $S_{\ell q_n}(\alpha,x)\neq \ell\mu_n\pm C_1 \frac{\ell^2}{c}$  for some $\ell=0,\ldots, k$, then there must exist $0\leq \ell\leq k$ s.t. $x, R_\alpha^{\ell q_n}(x)$ are separated by a discontinuity of $S_{q_n}(\alpha,\cdot)$. Since $\dist(x,R_\alpha^{\ell q_n}(x))\leq \ell/q_{n+1}$, $x$ must belong to a ball with radius $k/q_{n+1}$ centered at a discontinuity of $S_{q_n}(\alpha,\cdot)$.
By Lemma \ref{Lem-Sing-dist}, there are $d q_n$ discontinuities, so the measure of such points
is less than $d q_n \left(\frac{2k}{q_{n+1}}\right)\leq \frac{2dk}{c}$.  The lemma follows with $C_2:=2d$.
\end{proof}

\begin{lem}\label{Lem-max-S-r} There is a constant $C_3=C_3(b_1,\ldots,b_d)$ s.t. for every

\noindent
$\alpha=[0;a_1,a_2,\ldots]$,
$\max\{|S_r(\alpha,x)|\!: 0\leq r\leq q_n-1\}\leq C_3 (a_0+\cdots+a_{n-1}).$
\end{lem}
\begin{proof}
Let $r=\sum_{j=0}^{n-1} \fb_j q_j$ denote the Ostrowski expansion of $r$. Recall that this means that
$0\leq \fb_j\leq a_j$ and $\fb_j=a_j\Rightarrow \fb_{j-1}=0$.
So
$$
S_r=\sum_{k=0}^{\fb_{n-1}-1}S_{q_{n-1}}\circ R_\alpha^{q_{n-1}k}+
\sum_{k=0}^{\fb_{n-2}-1}S_{q_{n-2}}\circ R_\alpha^{q_{n-2}k}+\cdots+
\sum_{k=0}^{\fb_{0}-1}S_{q_{0}}\circ R_\alpha^{q_{0}k}\;.
$$
By the Denjoy-Koksma inequality
$
|S_r|\leq \sum \fb_j\mathsf{V}(f)\leq \mathsf V(f)\sum a_j
$ where $\mathsf V(f)\leq 2\sum b_i$ is the total variation of $f$ on $\Tor$.
\end{proof}

\begin{lem}
\label{LmDKV1}
There exist positive constants $\eps_1, \eps_2$ such that for every $\alpha$ irrational, if
$q_n$ is a principal denominator of $\alpha$ and $q_n r_n\in \cN$ with $r_n\leq M$
then
$\mes\{x:\; |S_{q_n}(\alpha,x)|\geq \eps_1\}\geq \eps_2 .$
\end{lem}
\begin{proof}
We follow an argument from \cite{Beck-Diophantine}.
Suppose $q_n$ is a principal denominator of $\alpha$ and $q_n r_n\in\cN$ for some $r_n\leq M$.
Let $N=q_n r_n.$  Since $f(x)=-\sum_{m=1}^d b_m \sum_{j=1}^\infty \frac{\sin(2\pi j(x+\beta_m))}{\pi j}$,  for each $j\in\N$
$$ \|S_N(\alpha, \cdot)\|_{L^2}^2\geq \frac{1}{\pi^2 j^2}
\int_0^1 \left(\sum_{m=1}^d b_m \sum_{k=0}^{N-1} \sin(2\pi j (x+k\alpha+\beta_m))\right)^2 \; dx. $$

Using the identities $\sum_{k=1}^N \sin(y+kx)=\frac{\cos(y+x/2)-\cos(y+(2N+1)x/2)}{2\sin(x/2)}$ and
$\cos A-\cos B=2\sin(\frac{A+B}{2})\sin(\frac{B-A}{2})$ we find that
\begin{align*}
&\|S_N(\alpha, \cdot)\|_{L^2}^2\geq \\
&\geq  \left(\frac{\sin (\pi N j\alpha)}{\pi j\sin(\pi j \alpha)}\right)^2
\int_0^1 \left(\sum_{m=1}^d b_m  \sin \left(2\pi \left(jx+
j{ \tfrac{{(N-1)}\alpha}{2}}\right)+2\pi j \beta_m\right)
\right)^2
  dx\\
&=\left(\frac{\sin (\pi N j\alpha)}{\pi j\sin(\pi j \alpha)}\right)^2
\int_0^1 \left(\sum_{m=1}^d b_m  \sin \left(2\pi (y+j \beta_m)\right)
\right)^2
  dy\\
& =\left(\frac{\sin (\pi N j\alpha)}{\pi j\sin(\pi j \alpha)}\right)^2 D(j\beta_1,\ldots,j\beta_m)\text{ with $D$ as in \eqref{D-function}}.
\end{align*}
We now take $j=N=r_n q_n$. The first term is bounded below because
$
\|N\alpha\|\leq M\|q_n\alpha\|\leq \frac{M}{q_{n+1}}\leq \frac{M}{a_{n+1} q_n}\leq \frac{M^2}{a_{n+1}N}= o(\tfrac{1}{N}),
$
so $\frac{\sin (\pi N^2 \alpha)}{\pi N\sin(\pi N \alpha)}\xrightarrow[n\to\infty]{}\pi^{-1}$. The second term is bounded below by $\epsilon_0$, because $N=q_n r_n\in \cN$. It follows that for all $n$ large enough, $\|S_{r_n q_n}(\alpha,\cdot)\|_2>\sqrt{\eps_0}/2\pi$.

For any $L^2$-function $\varphi$ and any $\heps>0$,
$$\|\varphi\|_{L^2}^2 \leq \|\varphi\|_{L^\infty}^2 \mes\{x: |\varphi(x)|\geq \heps\}+\heps^2.$$
Hence
$
\displaystyle \mes\{x: |\varphi(x)|\geq \heps\}\geq \frac{\|\varphi\|_{L^2}^2-\heps^2}{\|\varphi\|_{L^\infty}^2}.
$
We just saw that for all $n$ large enough, $\|S_{r_n q_n}(\alpha,\cdot)\|_2>\sqrt{\eps_0}/2\pi$, and
by the Denjoy-Koksma inequality
$\|S_{r_n q_n}(\alpha,\cdot)\|_{L^\infty}\leq M\mathsf V(f)$. So for some $\heps>0$ and  for all $n$ large enough,
$
\mes\{x: |S_{r_n q_n}(\alpha,x)|>\heps\}\geq \heps.
$

Looking at the inequality
$|S_{r_n q_n}(\alpha,x)|\leq \sum_{k=0}^{r_n-1} |S_{q_n}(\alpha, x+k q_n \alpha)|$, we see that if $|S_{r_n q_n}(\alpha,x)|\geq \heps$, then
$|S_{q_n}(\alpha, x+k q_n \alpha)|\geq \heps/M$ for some $0\leq k\leq M-1$.
So
for all $n$ large enough,
$
\mes\{x: |S_{q_n}(\alpha,x)|>\heps/M\}\geq \heps/M.
$
\end{proof}

\subsection{Proof of Theorem \ref{ThUniform}}

Let  $\Omega^*(\alpha)$ be the set of $x$ where the conclusion of Theorem \ref{ThUniform} holds.
$\Omega^*(\alpha)$ is $R_\alpha$-invariant and it is measurable by
Lemma \ref{Lemma-Omega-Measurable} in the appendix.
Therefore to show that $\Omega^*(\alpha)$ has  full measure, it suffices to show that it has positive measure.


Suppose $\alpha\in \cA$ and let $n_k\uparrow\infty$ be a sequence satisfying (\ref{OddDenom})  with $\cN$ given by \eqref{RelBigFC}.
There is no loss of generality in assuming that
$$
\frac {a_{n_k+1}}{a_0+\cdots+a_{n_k}}>k^3.
$$
So
$
q_{n_k+1}> k^3 L_k q_{n_k}$, where
$L_k:=a_0+\cdots + a_{n_k}.$

Recall that $\mu_{n_k}(x)=S_{q_{n_k}}(\alpha,x)$.
For all $k$ sufficiently large, there is a set $A_k$ of measure at least $\eps_2/2$
such that for all $x\in A_k$,
\begin{equation}\label{key-id}
S_{\ell q_{n_k}}(\alpha,x)=\ell\left(\mu_{n_k}(x)\pm \frac{C_1 \ell}{k^3 L_k}\right)
\text{ for all }\ell=0,1,\ldots, kL_k
\end{equation}
\begin{equation}
\label{key-idExt}
|\mu_{n_k}(x) |\geq \eps_1.
\end{equation}
This is because Lemma  \ref{Lem-mu-n}
says that the total measure of $x$ for which (\ref{key-id}) fails is $O(1/k^2)$
while \eqref{key-idExt} holds on the set of measure $\eps_2$
by Lemma \ref{LmDKV1}.

It follows that $\mes(\bigcap_{n>1}\bigcup_{k>n}A_k)\geq \eps_2/2$. Therefore  there exists $x$ which belongs to infinitely many $A_k$.
After re-indexing  $n_k$, we may assume that \eqref{key-id}, \eqref{key-idExt} are
satisfied for all $k\in \N.$
Henceforth, we fix such an $x$ and work with this $x$.
Let $$N_k(x):=k L_k q_{n_k}, \quad B_k(x):=k L_k |\mu_{n_k}(x)|,\quad
A_k(x):=\frac{1}{2}(\sgn(\mu_{n_k}(x))-1) B_k.$$
Any $n\leq N_k$ can be written uniquely in the form
$$n=l(n) q_{n_k}+r(n) \quad\text{with}\quad 0\leq l(n)\leq k L_k\text{ and }0\leq r(n)<q_{n_k}. $$
It is easy to see that
$\frac{l(n)}{k L_k}\xrightarrow[k\to\infty]{\text{dist}}\text{U}[0,1]\text{ as }
n\sim \text{U}(1,\ldots,N_k).$

Writing
$ S_n(\alpha,x)=S_{l(n) q_{n_k}}(\alpha,x)+S_{r(n)}(\alpha,x+\alpha l(n) q_{n_k})$
we obtain from \eqref{key-id} and Lemma \ref{Lem-max-S-r} that
$$S_n(\alpha,x)=l(n) \mu_{n_k} (x)+O(L_k) .$$
So $\frac{S_n(x)}{B_k}$ is asymptotically uniform on
$[0,1]$ when $\mu_{n_k}>0$, and $[-1,0]$ when $\mu_{n_k}<0$.
So $\frac{S_n(x)-A_k}{B_k}\xrightarrow[k\to\infty]{}\mathrm{U}[0,1]$, as $n\sim\mathrm{U}(1,\ldots,N_k(x))$.
\hfill$\Box$

\subsection{Proof of Theorem \ref{ThNonUniform}}
\label{SSNUAEx}
Let $\Omega(\alpha)$ denote the set of $x\in\T:=\R/\Z$ for which there are $B_N(x)\to\infty$ and $A_N(x)\in\R$ s.t.
\begin{equation}\label{assumption}
\frac{S_n(\alpha,x)-A_N(x)}{B_N(x)}\xrightarrow[N\to\infty]{\text{dist}}\text{U}[0,1],\text{ as }n\sim \text{U}(1,\ldots,N).
\end{equation}
This is a measurable set, see the appendix.
Assume by way of contradiction that $\text{mes}[\Omega(\alpha)]\neq 0$ for some $\alpha\in\cA$.

$\Omega(\alpha)$ is invariant under $R_\alpha(x)=x+\alpha\mod 1$ on $\T:=\R/\Z$. Since  $R_\alpha$ is ergodic, and $\Omega(\alpha)$ is measurable,   $\text{mes}[\Omega(\alpha)]=1$.

Since $\alpha\in\cA$,  there is an increasing sequence $n_k$ satisfying (\ref{OddDenom})  where $\cN$ is given by \eqref{RelBigFC}.
We can choose $n_k$ so that $q_{n_k} r_{n_k}\in \cN$ for $r_{n_k}\leq M$, and  $a_{n_k+1}>k^3 L_k$
where
$L_k:=a_0+\cdots + a_{n_k}.$
In particular,
$q_{n_k+1}> k^3 L_k q_{n_k}.$

Recall that $\mu_{n_k}(x):=S_{q_{n_k}}(\alpha, x)$. By Lemma \ref{LmDKV1}
we can choose $x$ such that for infinitely many $k$,
$|\mu_{n_k}(x)|\geq \eps_1.$ We will suppose that
$\mu_{n_k}(x)>0$ for infinitely many $k$;
the case where $\mu_{n_k}(x)<0$ for infinitely many $k$ is similar.

\medskip
\noindent
{\sc Claim 1.\/} It is possible to assume without loss of generality that
$\|B_{q_{n_k}}\|_\infty:=\sup_{x\in \Omega(\alpha)}|B_{q_{n_k}}(x)|\leq 3 C_3 L_k$ for all $k$
where $C_3$ is the constant from Lemma \ref{Lem-max-S-r}.

\medskip
\noindent
{\em Proof.\/} We claim that for every $x$ with (\ref{assumption}), $B_{q_{n_k}}(x)\leq 3 C_3 L_k$
for all $k$ large enough.
Otherwise, by Lemma \ref{Lem-max-S-r}, there are infinitely many $k$ s.t.   $B_{q_{n_k}}(x)>3\max\{|S_r(\alpha,x)|:r=0,\ldots,q_{n_k-1}\}$, whence
$|S_n(\alpha,x)/B_{q_{n_k}}|\leq \frac{1}{3}$ for all $0\leq n\leq q_{n_k}-1$.  In such circumstances, (\ref{assumption}) does not hold (the spread is not big enough).

Since $B_{q_{n_k}}(x)\leq 3 C_3 L_k$ for all $k$ large enough, there is no harm in replacing $B_{q_{n_k}}(x)$ in (\ref{assumption}) by $\min\{B_{q_{n_k}}(x), 3 C_3 L_k\}$.

\medskip
\noindent
{\sc Claim 2.\/} Fix $D>C=|\sum b_m|$, and let $E_k$ denote the set of $x\in\Omega(\alpha)$ s.t.
$
S_r(\alpha,x)=S_r(\alpha,R_\alpha^{\ell q_{n_k}} (x))\pm \frac{D \ell}{ q_{n_{k}+1}}$
 for all $0\leq \ell\leq B_{q_{n_k}}(x), 0\leq r<q_{n_k}-1
$.
Then $\text{mes}(E_k^c)\leq C_4 k^{-3}$.

\medskip
\noindent
{\em Proof.\/} If $x\not\in E_k$, then there are $0\leq \ell\leq B_{q_{n_k}}(x), 0\leq r<q_{n_k}-1$ s.t.
$$|S_r(\alpha,x)-S_r(\alpha,x+\ell q_{n_k}\alpha)|\geq \frac{D \ell}{q_{n_{k}+1}}.$$
By Lemma \ref{LmLip}, $\{x\}$, $\{x+\ell q_{n_k}\alpha\}$ are separated by a singularity of $S_r(\alpha,\cdot)$. So
$x$ belongs to a ball of radius $2\|\ell q_{n_k}\alpha\|$ centered at one of the $dq_{n_k}$ discontinuities of $S_{q_{n_k}}(\alpha,\cdots)$.
Thus
$
\text{mes}(E_k^c)\leq dq_{n_k}\cdot 2\|\ell q_{n_k}\alpha\|.
$
Now $\|\ell q_{n_k}\alpha\|\leq \ell\|q_{n_k}\alpha\|\leq \frac{\|B_{q_{n_k}}\|_\infty}{q_{n_k+1}}\leq
\frac{3 C_3 L_k}{q_{n_k+1}}\leq \frac{3C_3 }{k^3  q_{n_k}}$
by our choice of $n_k$. So $\text{mes}(E_k^c)\leq {C_4}/{k^3}$ with $C_4:=6dC_3$.

\medskip
\noindent
{\sc Claim 3.\/}
 Let $F_k$ denote the set of $x\in\Omega(\alpha)$ s.t.
$$
S_{\ell q_{n_k}}(\alpha,x)=\ell \left(\mu_{n_k}(x) \pm \frac{C_1\ell}{k^3 L_k}\right)
\text{ for all }0\leq \ell\leq B_{q_{n_k}}(x).
$$
Then $\text{mes}(F_k^c)\leq C_5 k^{-2}$. 

\medskip
\noindent
{\em Proof.\/} This follows from Lemma \ref{Lem-mu-n}.

\medskip
By Claims 2 and 3, and a Borel--Cantelli argument,  for a.e. $x$ there is $k_0(x)$ s.t. $x\in E_k\cap F_k$ for all $k\geq k_0(x)$.

Suppose $k\geq k_0(x)$, and let $N_k:=q_{n_k} B_{q_{n_k}}(x)$. Every $0\leq n\leq N_k-1$ can be uniquely represented as $n=\ell q_{n_k}+r$ with $0\leq \ell\leq B_{q_{n_k}}(x)-1$ and
$0\leq r\leq q_{n_k}-1$. Using the bound $\|B_{q_{n_k}}\|_\infty=O(L_k)$, we find:
\begin{align*}
&\frac{S_n(\alpha,x)-A_{q_{n_k}}(x)}{B_{q_{n_k}}(x)}=\frac{S_{\ell q_{n_k}}(\alpha,x)}{B_{q_{n_k}}(x)}+\frac{S_r(\alpha,R_\alpha^{\ell q_{n_k}}x)-A_{q_{n_k}}(x)}{B_{q_{n_k}}(x)}\\
&=\frac{S_{\ell q_{n_k}}(\alpha,x)}{B_{q_{n_k}}(x)}+\frac{S_r(\alpha,x)-A_{q_{n_k}}(x)+o(1)}{B_{q_{n_k}}(x)},\text{ because }x\in E_k\\
&=\frac{\ell(\mu_{n_k}(x)+o(1))}{B_{q_{n_k}}(x)}+\frac{S_r(\alpha,x)-A_{q_{n_k}}(x)+o(1)}{B_{q_{n_k}}(x)}, \text{ because $x\in F_k$}.
\end{align*}

If  $n\sim \mathrm{U}(0,\ldots, N_k-1)$, then $\ell,r$ are independent random variables, $\ell\sim\mathrm{U}(0,\ldots,B_{q_{n_k}}(x)-1)$ and $r\sim \mathrm{U}(0,\ldots,q_{n_k}-1)$.
Thus the distribution of
$\frac{\ell(\mu_{n_k}(x)+o(1))}{B_{q_{n_k}}(x)}$ is close to $\mathrm{U}[0,\mu_{n_k}(x)]$,  and the distribution of $\frac{S_r(\alpha,x)-A_{q_{n_k}}(x)+o(1)}{B_{q_{n_k}}(x)}$ converges to $\mathrm{U}[0,1]$
(because $x\in \Omega$).

Taking a subsequence such that $\mu_{n_k}(x)\to \breps>\eps_1$ we see that the random variables
$\displaystyle \frac{S_n(\alpha,x)-A_{q_{n_k}}(x)}{B_{q_{n_k}}(x)}$, where $n\sim \text{U}(0, \ldots, n_k-1)$,
converge in distribution to the sum of two independent uniformly distributed random variables.
This contradicts to \eqref{assumption}, because the sum of two independent uniform random variables is not uniform.\hfill$\Box$

\appendix
\section{Measurability concerns}
Let $\Omega(\alpha)$
denote the set of $x\in\T:=\R/\Z$ such that for some  $B_N(x)\to\infty$ and $A_N(x)\in\R$,
$
\frac{S_n(\alpha,x)-A_N(x)}{B_N(x)}\xrightarrow[N\to\infty]{\text{dist}}U[0,1],\text{ as }n\sim \text{U}(1,\ldots,N)
$
and let
$\Omega^*(\alpha)$ denote the set of $x\in\T$ such that along a subsequence $N_k(x)$ there exist
 some  $B_{N_k}(x)\to\infty$ and $A_{N_k}(x)\in\R$,
$
\frac{S_n(\alpha,x)-A_{N_k}(x)}{B_{N_k}(x)}\xrightarrow[k\to\infty]{\text{dist}}\text{U}[0,1],\text{ as }n\sim \text{U}(1,\ldots,N_k).
$
We make no assumptions on the measurability of  $A_N,B_N, N_k$ as functions of $x$.
The purpose of this section is to prove:
\begin{lem}\label{Lemma-Omega-Measurable}
$\Omega(\alpha)$ and $\Omega^*(\alpha)$
are measurable.
\end{lem}

The crux of the argument  is to  show that $A_N(x),B_N(x)$ can be replaced by measurable functions, defined in terms of the
percentiles of the random quantities $S_n(x,\alpha)$, $n\sim \text{U}(1,\ldots,N)$.

Recall that given $0<t<1$, the  {\em upper and lower $t$-percentiles} of a random variable $X$ are defined by
$$
\begin{aligned}
\chi^+(X,t)&:=\inf\{\xi: \Pr(X\leq \xi)> t\}\\
\chi^-(X,t)&:=\sup\{\xi: \Pr(X\leq \xi)< t\}
\end{aligned}
\ \ \ \ \ \ (0<t<1).
$$
Notice that $\Pr(X\leq \chi^+(X,t))\geq t$, $\Pr(X<\chi^-(X,t))\leq t$, and
$\Pr(\chi^-(X,t)<X<\chi^+(X,t))=0$. In case $X$ is {\em non-atomic} (i.e. $P(X=a)=0$ for all $a$), we can say more:

\begin{lem}\label{Lemma-Percentile}
Suppose $X$ is a non-atomic real valued random variable, fix $0<t<1$ and let $\chi^\pm_t:=\chi^\pm(X,t)$,  then
\begin{enumerate}[(a)]
\item $\Pr(X<\chi^+_t)=t$ and $\Pr(X<\chi^-_t)=t$;
\item $\forall\epsilon>0$, $\Pr(\chi^-_t-\epsilon<X<\chi^-_t), \Pr(\chi^+_t<X<\chi^+_t+\epsilon)$ are positive;
\item $\exists t_1<t_2$ s.t. $\chi^-_{t_1}<\chi^+_{t_2}$ and $\chi^-_{t_1},\chi^+_{t_2}$ have the same sign.
\end{enumerate}
\end{lem}
\begin{proof}Since $X$ is non-atomic,
$\Pr(X<\chi^+_t)=\Pr(X\leq \chi^+_t)\geq t$ and $\Pr(X<\chi^-_t)\leq t$. If $\chi^+_t=\chi^-_t$, part (a) holds. If $\chi^+_t>\chi^-_t$ then for all $h>0$ small enough $\chi^-_t+h<\chi^+_t-h$ whence
\begin{align*}
0&\leq \Pr(\chi^-_t<X<\chi^+_t)=\lim\limits_{h\to 0^+}\Pr(\chi^-_t+h<X<\chi^+_t-h)\\
&=\lim\limits_{h\to 0^+}\Pr(X<\chi^+_t-h)-\lim\limits_{h\to 0^+}\Pr(X\leq \chi^-_t+h)\leq t-t=0.
\end{align*}
Necessarily $\lim\limits_{h\to 0^+}\Pr(X<\chi^+_t-h)=t$ and $\lim\limits_{h\to 0^+}\Pr(X\leq \chi^-_t+h)=t$, which gives us
$\Pr(X<\chi^+_t)=t$ and $\Pr(X<\chi^+_t)=\Pr(X\leq \chi^+_t)=t$.

For (b) assume by contradiction that  $\Pr(\chi^-_t-\epsilon<X<\chi^-_t)=0$, then  for all $\chi^-_t-\epsilon<\xi<\chi^-_t$, $\Pr(X\leq \xi)=\Pr(X< \chi^-_t)=t$, whence
$\chi^-_t\leq \chi^-_t-\epsilon$, a contradiction. Similarly, $\Pr(\chi^+_t<X<\chi^+_t+\epsilon)=0$ is impossible.

To prove (c) note that since $X$ is non-atomic, either $\Pr(X>0)$ or $\Pr(X<0)$ is positive. Assume w.l.o.g. that $\Pr(X>0)\neq 0$. By non-atomicity, there are positive $a<b$ s.t. $\Pr(X\in (0,a))\neq 0$ and $\Pr(X\in (a,b))\neq 0$. Take $t_1:=\Pr(X<a)$ and $t_2:=\Pr(X<b)$.
\end{proof}

From now on fix a non-atomic random variable $Y$, and choose  $0<t_1<t_2<1$ as in Lemma \ref{Lemma-Percentile}(c) s.t. $\chi^-(Y,t_1)<\chi^+(Y,t_2)$ and $\sgn(\chi^-(Y,t_1))=\sgn(\chi^+(Y,t_2))$.

\begin{lem}\label{Lemma-measurable-constants}
Let $S_N$ be (possibly atomic) random variables s.t. for some $A_N\in \R$ and $B_N\to\infty$,   $\frac{S_N-A_N}{B_N}\xrightarrow[N\to\infty]{\text{dist}}Y$. Then
$\frac{S_N-A_N^\ast}{B_N^\ast}\xrightarrow[N\to\infty]{\text{dist}}Y$, where $A_N^\ast,B_N^\ast$ are the unique solution to
\begin{equation}\label{star-system}
\left\{
\begin{aligned}
A_N^\ast+B_N^\ast \chi^-(Y,t_1)&=\chi^-(S_N,t_1)\\
A_N^\ast+B_N^\ast \chi^+(Y,t_2)&=\chi^+(S_N,t_2).
\end{aligned}
\right.
\end{equation}
\end{lem}

\begin{proof}
Without loss of generality, $\chi^-(Y,t_1), \chi^+(Y,t_2)$ are both positive.

We  need the following fact (which is not automatic since $S_N$ are allowed to be atomic):
\begin{equation}\label{percentile-conv}
\lim_{N\to\infty}\Pr(S_N<\chi^-(S_N,t))=t\text{ for all }0<t<1.
\end{equation}
Indeed, given
$\epsilon>0$, let
$
\xi_N:=B_N \chi^-(Y,t-\epsilon)+A_N,
$
then
{\small $$\Pr(S_N<\xi_N)=
\Pr\left(\tfrac{S_N-A_N}{B_N}<\chi^-(Y,t-\epsilon)\right)
\xrightarrow[N\to\infty]{}\Pr(Y<\chi^-(Y,t-\epsilon))=t-\epsilon,$$}
by Lemma \ref{Lemma-Percentile}(a). So for all $N$ large enough, $\xi_N\leq \chi^-(S_N,t)$, whence $\liminf\Pr(S_N<\chi^-(S_N,t))\geq \lim\Pr(S_N<\xi_N)=t-\epsilon$. Since $\epsilon$ is arbitrary,  $\liminf\Pr(S_N<\chi^-(S_N,t))\geq t$.
 The other inequality $\limsup\Pr(S_N<\chi^-(S_N,t))\leq t$ is clear since $\Pr(S_N<\chi^-(S_N,t))\leq t$ for all $N$.

\medskip
With (\ref{percentile-conv}) proved, we proceed to prove that
\begin{equation}\label{target-limits}
 \frac{A_N^\ast-A_N}{B_N}\xrightarrow[N\to\infty]{}0\ \  \text{ and }\ \ \frac{B_N^\ast}{B_N}\xrightarrow[N\to\infty]{}1.
\end{equation}
It will then be obvious that $\frac{S_N-A_N}{B_N}\xrightarrow[N\to\infty]{\text{dist}}Y$ implies $\frac{S_N-A_N^\ast}{B_N^\ast}\xrightarrow[N\to\infty]{\text{dist}}Y$.

Define two affine transformations, $\vf_N(t)=\frac{t-A_N}{B_N}$ and $\vf_N^\ast(t)=\frac{t-A_N^\ast}{B_N^\ast}$. Notice that $(\vf_N^\ast)^{-1}(t)=A_N^\ast+B_N^\ast t$, so
 $(\vf_N^\ast)^{-1}(\chi^-(Y,t_1))=\chi^-(S_N,t_1)$,  by (\ref{star-system}). Since $B_N^\ast=\frac{\chi^+(S_N,t_2)-\chi^-(S_N,t_1)}{\chi^+(Y,t_2)-\chi^-(Y,t_1)}>0$,  $\vf_N^\ast$ is increasing. By \eqref{percentile-conv},
$
\Pr(\vf_N^\ast(S_N)<\chi^-(Y,t_1))=\Pr(S_N<\chi^-(S_N,t_1))\xrightarrow[N\to\infty]{}t_1.
$
So
\begin{align*}
t_1&=\lim_{N\to\infty}\Pr(\vf_N^\ast(S_N)<\chi^-(Y,t_1))\\
&=\lim_{N\to\infty}\Pr\bigl(\vf_N(S_N)<\vf_N[(\vf_N^\ast)^{-1}(\chi^-(Y,t_1))]\bigr)\\
&=\lim_{N\to\infty}\Pr\biggl(\vf_N(S_N)<\frac{B_N^\ast}{B_N}\left(\chi^-(Y,t_1)+\frac{A_N^\ast-A_N}{B_N}\right)\biggr).
\end{align*}

We claim  that this implies that
\begin{equation}\label{First-Limit}
\liminf_{N\to\infty} \frac{B_N^\ast}{B_N}\left(\chi^-(Y,t_1)+\frac{A_N^\ast-A_N}{B_N}\right)\geq \chi^-(Y,t_1).
\end{equation}
Otherwise, $\exists\epsilon$ s.t.
$\liminf\limits_{N\to\infty} \frac{B_N^\ast}{B_N}\left(\chi^-(Y,t_1)+\frac{A_N^\ast-A_N}{B_N}\right)<\chi^-(Y,t_1)-\epsilon$, so
\begin{align*}
t_1&=\liminf_{N\to\infty}\Pr\biggl(\vf_N(S_N)<\frac{B_N^\ast}{B_N}\left(\chi^-(Y,t_1)+\frac{A_N^\ast-A_N}{B_N}\right)\biggr)\\
&\leq \liminf_{N\to\infty}\Pr\biggl(\vf_N(S_N)<\chi^-(Y,t_1)-\epsilon\biggr)=\Pr(Y<\chi^-(Y,t_1)-\epsilon)\\
&=\Pr(Y<\chi^-(Y,t_1))-\Pr(\chi^-(Y,t_1)-\epsilon\leq Y<t_1)\\
&<t_1,\text{ by Lemma \ref{Lemma-Percentile}, parts (a),(b)}.
\end{align*}

Similarly, one shows that
\begin{equation}\label{Second-Limit}
\limsup_{N\to\infty} \frac{B_N^\ast}{B_N}\left(\chi^+(Y,t_2)+\frac{A_N^\ast-A_N}{B_N}\right)\leq \chi^+(Y,t_2).
\end{equation}

It remains to see that \eqref{First-Limit} and \eqref{Second-Limit} imply
 \eqref{target-limits}.
First we divide \eqref{First-Limit} by
\eqref{Second-Limit} to obtain
$$
\limsup_{N\to\infty} \frac{\chi^-(Y,t_1)+\frac{A_N^\ast-A_N}{B_N}}{\chi^+(Y,t_2)+\frac{A_N^\ast-A_N}{B_N}}\geq \frac{\chi^-(Y,t_1)}{\chi^+(Y,t_2)}.
$$
Since $x\mapsto\frac{a+x}{b+x}$ is strictly decreasing
 on $[0, \infty)$ when $a>b>0$, this implies that
\begin{equation}
\label{ANAN}
\limsup_{N\to\infty}\frac{A_N^\ast-A_N}{B_N}\leq 0.
\end{equation}
Looking at \eqref{First-Limit}, and recalling that $\chi^-(Y,t_1)>0$, we deduce that
\begin{equation}
\label{BBAst}
\liminf_{N\to\infty}\frac{B_N^\ast}{B_N}\geq 1.
\end{equation}
Next we look at the difference of \eqref{First-Limit} and \eqref{Second-Limit} and obtain
$$
\limsup_{N\to\infty}\frac{B_N^\ast}{B_N}\bigl(\chi^+(Y,t_2)-\chi^-(Y,t_1)\bigr)\leq \chi^+(Y,t_2)-\chi^-(Y,t_1),
$$
whence $\limsup (B_N^\ast/B_N)\leq 1$.  Together  with \eqref{BBAst}, this
proves that $B_N^\ast/B_N\xrightarrow[N\to\infty]{}1$. Substituting this in \eqref{First-Limit}, gives $\liminf \frac{A_N^\ast-A_N}{B_N}\geq 0$, which,  in view of \eqref{ANAN}, implies that
$\frac{A_N^\ast-A_N}{B_N}\xrightarrow[N\to\infty]{}0$.
This completes the proof of \eqref{target-limits},
and with it, the lemma.
\end{proof}

\begin{proof}
[Proof of Lemma \ref{Lemma-Omega-Measurable}]
We begin with the  measurability of $\Omega(\alpha)$.

Let $S_N(x)$ denote the random variable equal to $S_n(\alpha,x)$ with probability $\tfrac{1}{N}$ for each $1\leq n\leq N$.

We will apply  Lemma \ref{Lemma-measurable-constants} with $Y:=\text{U}[0,1]$, $S_N=S_n(x)$ and (say) $t_1:=\frac{1}{3}$, $t_2:=\frac{2}{3}$. It says that
$$
\Omega(\alpha)=\{x\in\T:\tfrac{S_N(x)-A_N^\ast(x)}{B_N^\ast(x)}\xrightarrow[N\to\infty]{\text{dist}}U[0,1]\},
$$
where
$A_N^\ast(x)$ and $B_N^\ast(x)$ are the unique solutions to \eqref{star-system}.
Since the percentiles of $S_N(x)$ are  measurable as functions of $x$,  $A_N^\ast(x),B_N^\ast(x)$ are measurable as functions of $x$.

We claim that  $\Omega(\alpha)=\Omega_1(\alpha)\cap\Omega_2(\alpha)$ where
\begin{align*}
\Omega_1(\alpha)&:=\bigcap_{\ell=1}^\infty\bigcup_{M=1}^\infty\bigcap_{N=M+1}^\infty
\left\{x\in\T: \frac{1}{N}\sum_{n=1}^N 1_{(2,\infty)}\left(\left|\tfrac{S_n(\alpha,x)-A_N^\ast(x)}{B_N^\ast(x)}\right|\right)<\frac{1}{\ell}\right\}\\
\Omega_2(\alpha)&=\bigcap_{t\in\Q\setminus\{0\}}
\left\{x\in\T:
\lim_{N\to\infty}\frac{1}{N}\sum_{n=1}^N e^{it \left(\frac{S_n(\alpha,x)-A_N^\ast}{B_N^\ast(x)}\right)}
=\E\left(e^{itY}\right)\right\}.
\end{align*}
This will prove the lemma, since the measurability of $A_N^\ast(\cdot),B_N^\ast(\cdot)$ implies the measurability of $\Omega_i(\alpha)$.

If $x\in\Omega(\alpha)$ then $x\in\Omega_1(\alpha)$ because $\Prob[|\frac{S_N(x)-A_N^\ast}{B_N^\ast}|>2]\xrightarrow[N\to\infty]{}0$, and $x\in \Omega_2(\alpha_2)$ because
$\E\bigl(e^{it\bigl(\frac{S_N(x)-A_N^\ast}{B_N^\ast}\bigr)}\bigr)\xrightarrow[N\to\infty]{}
\E\left(e^{itY}\right)$ pointwise.

Conversely, if $x\in \Omega_1(\alpha)\cap\Omega_2(\alpha)$ then it is not difficult to see  that
$\E\bigl(e^{it\bigl(\frac{S_N(x)-A_N^\ast}{B_N^\ast}\bigr)}\bigr)\xrightarrow[N\to\infty]{}
\E\left(e^{itY}\right)$
for {\em all} $t\in\R$. So $x\in\Omega(\alpha)$ by L\'evy's continuity theorem.
Thus $\Omega(\alpha)=\Omega_1(\alpha)\cap\Omega_2(\alpha)$, whence $\Omega(\alpha)$ is measurable.

The proof that $\Omega^\ast(\alpha)$ is measurable is similar. Enumerate $\Q\setminus\{0\}=\{t_n:n\in\N\}$, then $\alpha\in\Omega^\ast(\alpha)$ iff for every $\ell\in\N$ there exist $M\in\N$  s.t. for some  $N>M$
\begin{enumerate}[(1)]
\item $\frac{1}{N}\sum_{n=1}^N 1_{(2,\infty)}\left(\left|\tfrac{S_n(\alpha,x)-A_N^\ast(x)}{B_N^\ast(x)}\right|\right)<\frac{1}{\ell}$
\item $|\E\bigl(e^{it_n\bigl(\frac{S_N(x)-A_N^\ast(x)}{B_N^\ast(x)}\bigr)}\bigr)-\E\left(e^{it_nY}\right)|<\frac{1}{\ell}$ for all $n=1,\ldots,\ell$
\end{enumerate}
These are measurable conditions, because $A_N^\ast(\cdot), B_N^\ast(\cdot)$ are measurable. So $\Omega^\ast(\alpha)$ is measurable.
\end{proof}

\bibliographystyle{alpha}
\bibliography{Bounded-Type-Bib}{}

\end{document}